\documentclass{amsart} \usepackage{graphicx}
\usepackage{amsfonts,amsmath,amssymb,amsthm,enumerate,latexsym}

\usepackage{tikz} 
\usetikzlibrary{arrows} 
\usetikzlibrary{petri}
\usetikzlibrary{backgrounds} 
\usetikzlibrary{calc}
\tikzstyle{puntik}=[circle,draw=blue!50,fill=blue!20,thick]
\tikzstyle{sipka}=[->,shorten <=1pt,>=angle 90,semithick]
\tikzstyle{forced}=[->,shorten <=1pt,>=angle 90,semithick,dashed]
\tikzstyle{bigbad}=[line width=3pt] \tikzstyle{every
token}=[draw=blue!50,fill=blue!20,thick] \def\equals{=} 

\theoremstyle{plain} \newtheorem{theorem}{Theorem}
 \newtheorem{corollary}[theorem]{Corollary}
\newtheorem{observation}[theorem]{Observation}
\newtheorem{proposition}[theorem]{Proposition}

\theoremstyle{definition}

\newcommand\algA{{\mathbf A}} 
\newcommand\algB{{\mathbf B}}
\newcommand\A{{\mathbb A}} 
 
\newcommand\C{{\mathcal C}} 
\newcommand\Arrr{{\mathcal R}}

\newcommand\compNP{{\textsf{NP}}}
\newcommand\AND{\mathrel{\wedge}} 
\DeclareMathOperator\CSP{CSP}
\DeclareMathOperator\Inv{Inv} 
\DeclareMathOperator\Pol{Pol}
\newcommand\SDmeet{$\operatorname{SD}(\wedge)$}

\begin{document}

\title{CSP for binary conservative relational structures}

\author{Alexandr Kazda} 
\email{alex.kazda@gmail.com} 
\urladdr{} 
\address{{Department of Mathematics, 1326 Stevenson Center, Vanderbilt University,
Nashville, TN 37240, USA }} 

\thanks{The author would like to thank Libor Barto
  for motivating him, and the reviewer for pointing out how the final part of
the proof can be simplified. Supported by the Czech-Polish cooperation grant
7AMB13PL013 ``General algebra and applications'' and the GA\v CR project
13-01832S} 

\begin{abstract} We prove that whenever $\A$ is a 3-conservative relational
  structure with only binary and unary relations, then the algebra of
  polymorphisms of $\A$ either has no Taylor operation (i.e.~$\CSP(\A)$ is
  $\compNP$-complete), or it generates an \SDmeet{} variety (i.e.~$\CSP(\A)$
has bounded width).  \end{abstract}

\maketitle

\section{Introduction}

In the last decade, the study of the complexity of the constraint satisfaction
problem (CSP) has produced several major results due to universal algebraic
methods (see e.g.~\cite{BJK}, \cite{larosetesson-hardness} and
\cite{barto-kozik-bw}).

In this paper, we continue in this direction and look at clones of
polymorphisms of finite 3-conservative relational structures with relations of
arity at most two (a generalization of conservative digraphs). We show that
whenever such a structure admits a Taylor operation, its operational clone
actually generates a meet semidistributive (\SDmeet{}) variety, so the
corresponding CSP problem is solvable by local consistency checking (also known
as the bounded width algorithm). Since relational structures without Taylor
operations yield $\compNP$-complete CSPs, we obtain a rather simple dichotomy of
CSP complexity in this case.

There have been numerous papers published on the behavior of conservative
relational structures. We have been mostly building on three previous results:
First, Andrei Bulatov proved in \cite{bulatov-conservative} the dichotomy of
CSP complexity for 3-conservative relational structures, for which Libor Barto recently
offered a simpler proof in \cite{libor-conservative}. Meanwhile, Pavol Hell and
Arash Rafiey obtained a combinatorial characterization of all tractable
conservative digraphs \cite{hell-rafiey}, and have observed that all tractable
digraphs must have bounded width. 

\section{Preliminaries}
A \emph{relational structure} $\A$ is any set $A$ together with a family of
\emph{basic relations} $\Arrr=\{R_i:i\in I\}$ where $R_i\subset A^{n_i}$. We will call the number
$n_i$ the \emph{arity} of $R_i$. As usual, we will consider only 
finite structures (and finitary relations) in this paper. 

Let $R$ be an $m$-ary relation and $f:A^n\to A$ an $n$-ary operation. We say
that \emph{$f$
preserves $R$} if whenever we have elements $a_{ij}\in A$ such that
\begin{align*}
(a_{11},a_{12},&\dots,a_{1m})\in R,\\
(a_{21},a_{22},&\dots,a_{2m})\in R,\\
&\vdots\\
(a_{n1},a_{n2},&\dots,a_{nm})\in R,
\end{align*}
then we also have
\[
\left(f(a_{11},\dots,a_{n1}),\dots,f(a_{1m},\dots,a_{nm})\right)\in R.
\]
If $\Arrr$ is a set of relations, then we denote by $\Pol(\Arrr)$ the set of all
operations on $A$ that preserve all $R\in\Arrr$. On the other hand, if $\Gamma$
is a set of operations on $A$, then we denote by $\Inv(\Gamma)$ the set of all
relations that are preserved by each operation $f\in \Gamma$.

One of the most important notions in CSP is the primitive positive definition. If
we have relations $R_1,\dots,R_k$ on $A$, then a relation $S$ on $A$ is
\emph{primitively positively defined} using $R_1,\dots, R_k$ if
there exists a logical formula defining $S$ that uses only conjunction, existential
quantification, symbols for variables, predicates $R_1,\dots,R_k$, and
the symbol for equality ``$=$''.

Observe that the set $\Pol(\Arrr)$ is closed under composition and contains all
the projections, therefore it is an \emph{operational clone}. If $\A=(A,\Arrr)$
is a relational structure, then $\Inv(\Pol(\Arrr))$ consists of precisely all
the relations that can be primitively positively defined using the relations
from $\Arrr$ (see the original works of Bodnarchuk~\cite{bodnarchuk} and
Geiger~\cite{geiger}, or the survey \cite{galois} for proof of this statement,
as well as a more detailed discussion of the correspondence between $\Pol$ and
$\Inv$). We will call $\Inv(\Pol(\Arrr))$ the \emph{relational clone} of $\A$. 

Given an algebra $\A$, an instance of the \emph{constraint satisfaction
problem} $\CSP(\A)$ consists of a set of variables $V$ and a set of constraints
$\C$ where each constraint $C=(S,R)$ has a \emph{scope} $S\subset V$ and a
\emph{relation} $R\subset A^S$ such that $R$ (after a suitable renaming of
variables) is either equality, or one of the basic relations of $\A$. A
\emph{solution} of this instance is any mapping $f:V\to A$ such that $f_{|S}\in
R$ for each constraint $(S,R)\in\C$. In this paper we will only consider CSP
instances where all relations have scopes of size at most two.  

We can draw a CSP instance as a constraint network (also known as potato
diagram): For each variable $x$ we have the potato $B_x\subset A$ equal to the
intersection of all the unary constraints on $x$. For each constraint
$(\{x,y\},R)$ of arity two we draw lines from elements of $B_x$ to elements of
$B_y$ that correspond to the relation $R$.

To solve the instance now means to choose in each potato $B_x$ a vertex $b_x$
so that whenever $C=(\{x,y\},R)$ is a constraint, there is a line in $R$ from
$b_x$ to $b_y$ (i.e. $(b_x,b_y)\in R$). See Figure~\ref{figPotato} for an example.

If we mark some variables in the CSP instance $I$ as free variables and print
out values of these variables in all solutions of $I$, we obtain a relation on
$A$. It turns out that there is a straightforward correspondence between CSP
instances with free variables and primitive positive definitions. See
Figure~\ref{figExample} for an example of such a correspondence.

\begin{figure}
  \begin{center}
    \begin{tikzpicture}[bend angle=45,draw,auto]
      \node[puntik] (a){};
      \node (fakea)[below of=a]{};
      \node[puntik,draw=black,very thick] (b)[below of=fakea]{};
      \node(Bx)[below of=b, node distance=1.5cm]{$B_x$};
      \node[puntik,draw=black,very thick]      (c)       [below right of=b,node distance=3cm]  {};
        \node[puntik]      (d)       [right of=c] {};
	\node[puntik]      (e)      [right of=d] {};
      \node(By)[below right of=e, node distance=1.5cm]{$B_y$};
	\node[puntik,draw=black,very thick] (f) [above right of=e, node distance=3cm]{};
	\node[puntik] (g)[above of=f] {};
	\node[puntik] (h)[above of=g]{};
      \node(Bz)[below right of=f, node distance=1.5cm]{$B_z$};

	\draw (d.center) ellipse (2cm and 1 cm);
	\draw[rotate=-10] (fakea) ellipse (1cm and 2 cm);
	\draw[rotate=10] (g) ellipse (1cm and 2 cm);
	\path
	(a) edge (h)
	edge (g)
	(b) edge[bigbad] (f)
	(c) edge[bigbad] (b)
	    edge[bigbad] (f)
	(d) edge (h)
 	    edge (f)
	(e) edge (f);

  \end{tikzpicture}
  \end{center}
  \caption{An example of a potato diagram with three variables $x,y,z$
    and three binary relations (instance solution in bold)}
  \label{figPotato}
\end{figure}
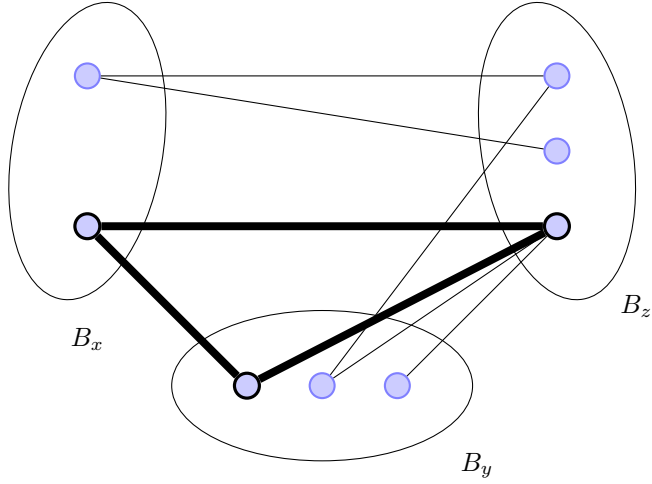

\begin{figure}
  \begin{center}
  \begin{tikzpicture}[auto,draw, node distance=7mm]
    \node[label=below left:$B_x$] (x0) {};
    \node[right of=x0,label=below:$B_s \equals R_1$, node distance=3cm] (s0) {};
    \node[right of=s0,label=below right:$B_y$, node distance=3cm] (y0) {};
    \foreach \i/\j in {0/1,1/2,2/3,3/4,4/5}
    {
    \node[above of=x\i,puntik] (x\j){};
    \node[above of=y\i,puntik] (y\j){};
  }
   \foreach \i/\j in {0/1,1/2,2/3}
    \node[above of=s\i,puntik] (s\j){};
    \draw [thick,rotate=-5] (x3) ellipse (1cm and 2.2cm);
    \draw [thick,rotate=5] (y3) ellipse (1cm and 2.2cm);
   \draw (s2) ellipse (1cm and 1.5cm);
   \node[left of=s0,node distance=1.5cm]{$R_2$};
   \node[right of=s0,node distance=1.5cm]{$R_2$};
   \foreach \u/\v in {5/1,3/1,3/2,5/2,1/1,4/3}
   {
     \path 
     (x\u) edge (s\v)
     (y\u) edge (s\v);
   }
  \end{tikzpicture}
  \end{center}
  \caption{Constraint network with the free variables $x,y$ which defines
   the relation $S=\{(x,y):\exists s,(s)\in R_1\AND (x,s)\in R_2\AND (y,s)\in R_2\}$.}
  \label{figExample}
\end{figure}
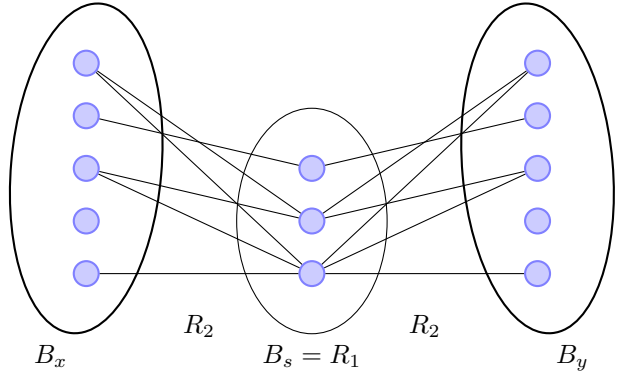

Let $\algA$ be an algebra. We say that the variety generated by
$\algA$ is \emph{congruence meet semidistributive} (\SDmeet{} for short) if for any
algebra $\algB$ in the variety generated by $\algA$ and any congruences
$\alpha,\beta,\gamma$ in $\algB$ we have
\[
\alpha\AND\beta=\alpha\AND\gamma \Rightarrow
\alpha\AND(\beta\vee\gamma)=\alpha\AND\beta.
\]

Given $\A$, if the (2,3)-consistency checking algorithm as defined
in~\cite{barto-kozik-bw} always returns a correct answer to any instance of
$\CSP(\A)$ (i.e. if there are no false positives), we say that $\A$ has
\emph{bounded width}. The following result shows a deep connection between
bounded width and congruence meet semidistributivity.

\begin{theorem}\label{thm-bw}
Let $\A$ be a finite relational structure containing all the one-element unary
relations (constants). Then the following are equivalent:
\begin{enumerate}
\item $\A$ has bounded width,
\item the variety generated by $(A,\Pol(\A))$ is congruence meet
semidistributive,
\item $\Pol(\A)$ contains ternary and quaternary weak near unanimity (WNU) operations
with the same polymer, i.e. there exist idempotent $u,v\in \Pol(\A)$ such that for all
$x,y\in A$ we have:
\[
u(x,x,y)=u(x,y,x)=u(y,x,x)=v(y,x,x,x)=\dots=v(x,x,x,y)
\]
\end{enumerate}
\end{theorem}
\begin{proof}
For ``$1\Rightarrow 3$'', see the upcoming survey~\cite{manuscript}, while
``$2 \Rightarrow 1$'' is the main result of~\cite{barto-kozik-bw}.

To prove ``$3\Rightarrow 2$'', it is enough to observe that the equations for
idempotent ternary and quaternary WNU operations with the same polymer fail in any nontrivial
variety of modules. Therefore, as shown in~\cite[Theorem 9.10]{chapter9}, the third
condition implies congruence meet semidistributivity (this is true even 
in the case of infinite algebras, as shown in~\cite{kearnes-kiss-shape}). 
\end{proof}

We say that an $n$-ary operation $t$ is \emph{Taylor} if for every $1\leq k\leq n$ the
opearation $t$ satisfies some equation of the form
\[
  t(u_1,\dots,u_{k-1},x,u_{k+1},\dots,u_n)\approx t(v_1,\dots,v_{k-1},y,v_{k+1},\dots,v_n),
\]
where the
(different) variables $x$ and $y$ are both on the $k$-th place (this is the
weakest set of equations that no projection can satisfy). An algebra
$\algA$ (resp. a relational structure $\A$) admits a Taylor operation if there is a
Taylor operation in the operational clone of $\algA$ (resp. in $\Pol(\A)$).
If $\A$ is a relational structure with
all the constants (one-element unary relations) that does not
admit any Taylor operation, then $\CSP(\A)$ is known to be $\compNP$-complete
(see~\cite[Corollary 7.3]{BJK} together with~\cite[Lemma 9.4]{chapter9}).

A relational structure $\A$ is \emph{conservative} if $\A$ contains all the
possible unary relations on $A$. We will call a relational
structure $\A$ \emph{3-conservative} if $\A$ contains all the one, two and
three-element unary relations.

\section{Red, yellow, and blue pairs}
Assume that $\A$ is a 3-conservative relational structure that admits 
a Taylor operation. Then for every pair of vertices $a,b\in A$ there must exist a polymorphism
of $\A$ that, when restricted to $\{a,b\}$, is a semilattice, majority,
or minority. If there was a pair without such a polymorphism, then a result by Schaefer
\cite{schaefer} implies that all the operations in $\Pol(\A)$ restricted to $\{a,b\}$
are projections, and so $\Pol(\A)$ can not contain a Taylor term. 

We will color each pair $\{a,b\}\subset A$ as follows:
\begin{enumerate}
\item If there exists $f\in\Pol(\A)$ semilattice on $\{a,b\}$, we color
$\{a,b\}$ \emph{red}, else
\item if there exists $g\in\Pol(\A)$ majority on $\{a,b\}$, we color $\{a,b\}$
\emph{yellow}, else
\item if there exists $h\in\Pol(\A)$ minority operation on $\{a,b\}$, we color $\{a,b\}$
\emph{blue}.
\end{enumerate}

In \cite{bulatov-conservative}, Andrei Bulatov proves the Three Operations
Proposition (we change the notation to be compatible with ours and omit the
last part of the proposition which we will not need):

\begin{theorem}\label{thmBulatov}
  Let $\A$ be a 3-conservative relational structure. There are 
  polymorphisms $f(x,y)$, $g(x,y,z)$, and $h(x,y,z)$ of $\A$ such that for
every two-element subset $B\subset A$:
\begin{itemize}
\item $f_{|B}$ is a semilattice operation whenever $B$ is red, and
$f_{|B}(x,y)=x$ otherwise,
\item $g_{|B}$ is a majority operation if $B$ is yellow; $g_{|B}(x,y,z)=x$ if
$B$ is blue, and $g_{|B}(x,y,z)=f_{|B}(f_{|B}(x,y),z)$ if $B$ is red
\item $h_{|B}$ is a minority operation if $B$ is blue; $g_{|B}(x,y,z)=x$ if
$B$ is yellow, and $g_{|B}(x,y,z)=f_{|B}(f_{|B}(x,y),z)$ if $B$ is red.
\end{itemize}
\end{theorem}

We omit the proof of the theorem here, but we note that it turns out that the
operations $f,\,g$, and $h$ can be obtained in a straightforward way by
patiently composing terms (also, one actually does not need the full power of
3-conservativity here; 2-conservativity would suffice). 

\begin{corollary}
If $\A$ is such that all its pairs are red or yellow, then $\A$ has bounded width
since the operations
\begin{align*}
u(x,y,z)&= g(f(f(x,y),z),f(f(y,z),x),f(f(z,x),y))\\
v(x,y,z,t)&=g(f(f(f(x,y),z),t),f(f(f(y,z),x),t)),f(f(f(z,x),y),t)) 
\end{align*}
are a pair of ternary and quaternary WNUs with the same polymer. If $x,y$ are red then
$u(x,x,y)=v(x,x,x,y)=f(x,y)$, and if $x,y$ are yellow, then 
$u(x,x,y)=v(x,x,x,y)=x$.
\end{corollary}

By Theorem~\ref{thm-bw}, it is enough to show that if $\CSP(\A)$ is not
$\compNP$-complete, then $\A$ does not have any blue pair of vertices.

We could end our paper at this point and refer the reader to the article
\cite{hell-rafiey} which, among other things, shows by combinatorial methods
that if $G$ is a conservative digraph and $\CSP(G)$ is not $\compNP$-complete,
then all pairs of vertices of $G$ are either yellow or red. However, we would
like to present a short algebraic proof of this statement and generalize it
beyond digraphs.

\section{Main proof}
We proceed by contradiction. Let us for the remainder of this section fix a
3-conservative relational structure $\A$ (with
unary and binary relations only) that admits a Taylor term,
yet there exists a blue pair $\{a,b\}\subset A$.

\begin{proposition}\label{propRS}
The relational clone of $\A$ contains the 
relation 
\[
  R=\{(b,b,b),(a,a,b),(a,b,a),(b,a,a)\}.
\]
\end{proposition}
\begin{proof}
Consider the ternary relation 
\[
R=\{(t(a,a,b),t(a,b,a),t(b,a,a)): t\in \Pol_3(\A)\},
\]
where $\Pol_3(\A)$ denotes all the ternary polymorphisms of $\A$. It is easy to
see that the relation $R$ lies in $\Inv(\Pol(\A))$. Since the projections
$\pi_1,\pi_2,\pi_3$, and the minority $h$ belong to $\Pol_3(\A)$, substituting
these polymorphisms for $t$ yields that
$(a,a,b),(a,b,a),(b,a,a)$, and $(b,b,b)$ lie in $R$. 

Since $\{a,b\}\leq \algA$, we have $R\subset\{a,b\}^3$. Assume now that $R$ contains 
more than the four elements given above. If $(a,a,a)\in R$, then there exists $t\in \Pol_3(\A)$ that acts as a
majority on $\{a,b\}$, and $a,b$ should have been yellow. If $(b,b,a)\in R$, then
there exists some $t$ such that 
\begin{align*}
t(a,a,b)&=t(a,b,a)=b\\
t(b,a,a)&=a.
\end{align*}
Since $\{a,b\}$ is not red, $t(b,b,a)=t(b,a,b)=a$ and $t(a,b,b)=b$ (otherwise 
one of $t(x,x,y)$, $t(x,y,x)$, or $t(y,x,x)$ would be a semilattice operation). But then
$f(x,y,z)=h(t(x,y,z),t(y,z,x),t(z,x,y))$ is a majority operation on $\{a,b\}$, a contradiction.

We can handle the cases $(b,a,b),(a,b,b)\in R$ in a similar fashion.
\end{proof}


If $(x,y,z)$ is a triple of free variables of the CSP instance
$I$ and $s$ is a solution of $I$, then we say that $s$ \emph{realizes} the triple 
$(c,d,e)\in A^3$ if $s(x)=c$, $s(y)=d$, and $s(z)=e$. We say that $I$ realizes
some triple of elements if there exists a solution $s$ of $I$ that realizes
this triple.

Since $R$ lies in the relational clone of $\A$, it follows that $R$ lies in the
relational clone of the structure $\overline{\A}$ obtained from $\A$ by adding all the
unary and binary relations in the relational clone of $\A$ as basic relations.
(We will need to use these relations to make our induction work.) Hence, there 
is a $\CSP(\overline{\A})$ instance $I$
and three variables $x$, $y$, $z$ such that $I$ with free variables $x$, $y$, $z$ realizes precisely all the triples in $R$.  Let $\{B_j: j\in J\}$ be
the potatoes in the constraint network of $I$. Choose $I$ so that the sum of
the sizes of its potatoes is minimal among all possible $\CSP(\overline{\A})$ instances
realizing $R$.

Observe that if $s_1,s_2,s_3$ are solutions of $I$ and $p$ is a ternary operation
preserving all the relations used in $I$, then $p(s_1,s_2,s_3)$ is also a solution of $I$.

\begin{observation}\label{obsFour}
For every $j\in J$, we have $|B_j|$ equal to 2 or 3.
\end{observation}

\begin{proof}
It is easy to see that the potatoes for $x,y,$ and $z$ are all equal to
$\{a,b\}$.

Assume that there is a potato $B_j$ with at least four distinct elements. Let
$s_{aab},s_{aba},$ and $s_{baa}$
be solutions of $I$ realizing the triples $(a,a,b)$, $(a,b,a)$, and $(b,a,a)$.

Now let $B_j'=\{s_{aab}(j),s_{aba}(j),s_{baa}(j)\}$. If we replace $B_j$ by 
$B_j'$ (using a unary constraint), we get a smaller
instance $I'$ which keeps the solutions $s_{aab},s_{aba}$ and $s_{baa}$.
We know that $\overline{\A}$ has a polymorphism $h$ which is the minority on
$\{a,b\}$. Therefore, $h(s_{aab},s_{aba},s_{baa})$ is a solution of $I'$ that
realizes the triple $(b,b,b)$. The
instance $I'$ then realizes precisely all the elements of $R$, a contradiction with the minimality of $I$.

If some $B_j$ was a singleton, we could simply remove the variable $j$ from
$I$. The potatoes that were connected with $B_j$ by binary constraints might
need to become smaller, but that is easy to achieve using unary constraints. We
would obtain in this way a smaller instance that still realizes $R$.  
\end{proof}

\begin{observation}\label{obsBlue}
The pair $\{c,d\}$ is blue for every distinct $c,d\in B_j,\,j\in J$
\end{observation}

\begin{proof}
Assume first that $\{c,d\}$ is red. Let without loss of generality $f(c,d)=d$, where $f$ is the
semilattice-like polymorphism
from Theorem~\ref{thmBulatov}. We know that there exists a
solution $s$ of $I$ such that $s(j)=d$ (otherwise, we could just delete
$d$ from $B_j$). If now $r$ is a solution such that $r(j)=c$, then $f(r,s)$ is
also a solution of $I$, and $f(r,s)(j)=d$. What is
more, since $f=\pi_1$ on $\{a,b\}$, the solution $f(r,s)$ realizes the same
triple as $r$. Therefore, we can remove $c$ from $B_j$ without losing anything
in $R$.

The situation for $\{c,d\}$ yellow is similar. Again, let $s, r$ be
solutions such that $s(j)=d$ and $r(j)=c$. Then $g(r,s,s)$ is a solution that
realizes the same triple as $r$ and satisfies $g(r,s,s)(j)=d$. We can thus 
eliminate $c$ from $B_j$.
\end{proof}

We note that one of the main ingredients in the above proof was the fact that the algebra
$\{d\}$ absorbs $\{c,d\}$ in the sense of~\cite{barto-kozik-bw}. 

\begin{observation}\label{obsTwo}
For every $j\in J$, we have $|B_j|= 2$.
\end{observation}

\begin{proof}
Assume that we have a $j$ such that $B_j=\{c,d,e\}$.
As in the proof of Observation~\ref{obsFour}, let $s_{aab},s_{aba},s_{baa}$
be some solutions of $I$ realizing $(a,a,b)$, $(a,b,a)$, $(b,a,a)$.

If $\{s_{aab}(j),s_{aba}(j),s_{baa}(j)\}\neq B_j$, then we can make $B_j$ (and
therefore $I$) smaller like in the proof of Observation~\ref{obsFour}. Without
loss of generality assume that
\begin{align*}
  s_{baa}(j)&=c,\\
  s_{aba}(j)&=d,\\
  s_{aab}(j)&=e,\\
  s_{bbb}(j)&=c.
\end{align*}

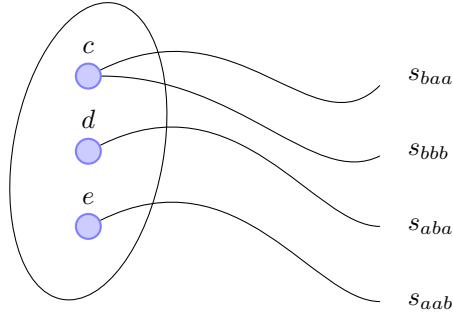
\begin{figure}
  \begin{center}
    \begin{tikzpicture}[bend angle=45,draw,auto]
      \node[puntik,label=$c$] (c){};
      \node[puntik, below of=c,label=$d$] (d){};
      \node[puntik, below of=d,label=$e$] (e){};
      \node[label=right:$s_{baa}$, right of=c,node distance=4cm](sbaa){};
      \node[label=right:$s_{bbb}$,below of=sbaa](sbbb){};
      \node[label=right:$s_{aba}$,below of=sbbb](saba){};
      \node[label=right:$s_{aab}$, below of=saba](saab){};
      \draw[rotate=-10] (d) ellipse (1cm and 2 cm);
      \draw(c)..controls ($(c)+(2cm,1cm)$) and ($(c)+(3cm,-1cm)$).. (sbaa);
      \draw(d)..controls ($(d)+(2cm,1cm)$) and ($(d)+(3cm,-1cm)$).. (saba);
      \draw(e)..controls ($(e)+(2cm,1cm)$) and ($(e)+(3cm,-1cm)$).. (saab);
      \draw(c)..controls ($(c)+(2cm,0)$) and ($(c)+(3cm,-1.5cm)$).. (sbbb);
  \end{tikzpicture}
  \end{center}
  \caption{Proving Observation~\ref{obsTwo}}
  \label{figTwo}
\end{figure}

Now consider the solution $r=h(s_{baa},s_{aba},s_{bbb})$
where $h$ again comes from Theorem~\ref{thmBulatov}.
Since all the pairs in $B_j$ are blue, $r$ is a realization of $(a,a,b)$ such
that $r(j)=h(c,d,c)=d$. This means that we can safely delete $e$ from $B_j$ and
again get a smaller instance that realizes $R$.
\end{proof}

Let us put together what we know about $I$: We have two-element potatoes
everwhere, connected by binary constraints. Now observe that every binary
relation on a two element set is invariant under the majority map $m$. Therefore,
$I$ must realize the triple
\[
m((b,a,a),(a,b,a),(a,a,b))=(a,a,a),
\]
a contradiction.

We state our result as a theorem:

\begin{theorem}\label{thmTaylorSDmeet}
If $\A$ is a finite 3-conservative relational structure that admits a Taylor
term and all of its basic relations are binary or unary, then the variety generated
by $(A,\Pol(\A))$ is \SDmeet{}.
\end{theorem}

Translating our result to the CSP complexity setting, we obtain:

\begin{corollary}[Dichotomy for 3-conservative CSPs with binary relations]
If $\A=(A, R_1,\dots,R_k)$ is a finite 3-conservative relational structure that admits a Taylor
term and all of its basic relations are binary or unary, then $\CSP(\A)$ has bounded
width. If $\A$ does not admit a Taylor polymorphism, then $\CSP(\A)$ is
$\compNP$-complete.
\end{corollary}

Note, however, that our result can not be generalized to relational structures
with ternary basic relations. For example the structure $\A$ on $\{0,1\}$ with all
the unary relations together with the relation $S=\{(x,y,z):x+y+z=0\pmod 2\}$ has
$\Pol(\A)$ consisting of all the idempotent linear mappings over ${\mathbb Z}_2$, so
$\Pol(\A)$ can not contain any quaternary WNU.

Our result is also false for 2-conservative relational structures with arity of
all relations at most two. Let $\A$ be the relational structure on $\{0,1\}^2$ with 
all unary relations of size one and two, plus the 
three equivalence relations $\alpha$, $\beta$, and $\gamma$. These three
equivalences correspond to the following partitions of $\{0,1\}^2$:
\begin{align*}
\alpha\dots&\{\{(0,0),(1,1)\},\{(0,1),(1,0)\}\},\\
\beta\dots&\{\{(0,0),(0,1)\},\{(1,0),(1,1)\}\},\\
\gamma\dots&\{\{(0,0),(1,0)\},\{(0,1),(1,1)\}\}.
\end{align*}
Observe that $\Pol(\A)$ contains the
idempotent Taylor term $p(x,y,z)=x+y+z$, where addition is taken componentwise
and modulo 2 (i.e. like in ${{\mathbb Z}_2}^2$). However,
the variety generated by $\algA=(\{0,1\}^2,\Pol(\A))$ is definitely not \SDmeet{} since
$\alpha,\beta,\gamma$ are congruences of $\algA$ 
such that 
 $\alpha\AND\beta=\alpha\AND\gamma=0_{\algA},$ while
$\alpha\AND(\beta\vee\gamma)=\alpha$.

\section{Closing remarks}
It is remarkably difficult to obtain a digraph that would have a tractable CSP,
yet it would not have bounded width (though such beasts do exist; see the
original argument in~\cite{atserias-bw}, or the construction in
\cite{todd-et-al}). We give a partial explanation for this phenomenon: such
digraphs needs to admit some nonconservative binary or ternary operation, while
avoiding ternary and quaternary WNUs.

As we have seen, our result about 3-conservative binary structures is quite
tight. However, some generalizations might still be possible. At the moment, we
do not know if our result holds for 2-conservative digraphs and if
Theorem~\ref{thmTaylorSDmeet} holds when we drop the finiteness condition. We suspect
that the answer to both questions will be negative, but the counterexamples
might turn out to be illuminating.

\bibliographystyle{spmpsci}
\bibliography{citations}
\end{document}